\documentclass[11pt]{amsart}

\usepackage{latexsym}
\usepackage{amssymb}
\usepackage{amsmath}
\usepackage{amsfonts}
\usepackage[mathscr]{eucal}
\usepackage{MnSymbol}
\usepackage{color}

\usepackage{amscd}

\usepackage{graphicx}
\usepackage{epsfig}

\usepackage{amsthm}
\usepackage[all,cmtip]{xy}
\usepackage[pdfstartview=FitH]{hyperref}
\usepackage{pst-grad} 
\usepackage{pst-plot} 
\usepackage{comment}
\usepackage{psfrag}
\usepackage{verbatim}
\usepackage{hyperref}
\hypersetup{
colorlinks,
allcolors=blue
}

\usepackage{soul}

\newtheorem{thm}{Theorem}[section]
\newtheorem{cor}[thm]{Corollary}
\newtheorem{lem}[thm]{Lemma}
\newtheorem{prop}[thm]{Proposition}

\newtheorem{theorem}{Theorem}

\theoremstyle{definition}
\newtheorem{definition}[thm]{Definition}

\newtheorem{rem}[thm]{Remark}

\newtheorem*{ack}{Acknowledgements}

\numberwithin{equation}{section}

\DeclareMathOperator{\Susp}{Susp}

\newcommand{\RP}{\mathbb{R}P}

\newcommand{\sphere}{\mathrm{\mathbb{S}}}

\newcommand{\Real}{\mathrm{\mathbb{R}}}

\newcommand{\SO}{\mathrm{SO}}

\begin{document}



\title[Three-dimensional Alexandrov spaces]{Three-dimensional Alexandrov spaces with local isometric circle actions  
}

\author[F.~Galaz-Garc\'ia]{Fernando Galaz-Garc\'ia$^*$}

\author[J.~N\'u\~nez-Zimbr\'on]{Jes\'us N\'u\~nez-Zimbr\'on $^{\ast\ast}$}

\thanks{$^*$ Partially supported by DGAPA-UNAM grant PAPIIT IN-113516.}
\thanks{$^{\ast\ast}$ Partially supported by PAEP-UNAM, DGAPA-UNAM grant PAPIIT IN-113516 and a UC MEXUS-CONACYT postdoctoral grant under the project ``Alexandrov geometry''.}

\address[F.~Galaz-Garc\'ia]{Institut f\"ur Algebra und Geometrie, Karlsruher Institut f\"ur Technologie (KIT), Karlsruhe, Germany}
\email{galazgarcia@kit.edu}

\address[J.~N\'u\~nez-Zimbr\'on]{Departament of Mathematics, University of California, Santa Barbara, USA}
\email{zimbron@ucsb.edu}

\date{\today}




\begin{abstract}
We obtain a topological and  equivariant classification of closed, connected three-dimensional Alexandrov spaces admitting a local isometric circle action. 
We show, in particular, that such spaces are homeomorphic to connected sums of some closed $3$-manifold with a local circle action and finitely many copies of the suspension of the real projective plane.
\end{abstract}
\maketitle




\section{Introduction}

Alexandrov spaces are metric generalizations of Riemannian manifolds with (sectional) curvature bounded  below; they were first studied in \cite{BGP} and have provided a natural setting to study questions of global Riemannian geometry. Although the role of these spaces in Riemannian geometry has been one of the main motivations for their study, they are also objects of intrinsic interest.

As for Riemannian manifolds, one may investigate Alexandrov spaces via their symmetries. Since the isometry group of a compact Alexandrov space is a compact Lie group (see \cite{FY}),  the symmetry point of view naturally leads to  the study of  isometric Lie group actions on Alexandrov spaces. In this context, the  combined work of Berestovski\v{\i} \cite{Be}, Galaz-Garc\'ia and Searle \cite{GS} and N\'u\~nez-Zimbr\'on \cite{NZ} yields equivariant and topological classifications of closed Alexandrov spaces of dimension at most three with an effective, isometric action of a compact, connected Lie group. Here, as for manifolds, an Alexandrov space is said to be \emph{closed} if it is compact and has no boundary.

In the present article we focus our attention on closed three-dimensional Alexandrov spaces with \emph{local circle actions}. These actions generalize isometric circle actions and are decompositions of the space into disjoint, simple, closed curves (possibly single points) such that each one of these curves has a small tubular neighborhood $V$ equipped with a circle action. The orbits of this circle action are required to be the curves of the decomposition contained in $V$. 

In the context of topological manifolds, Orlik, Raymond \cite{OR2}, and Fintushel \cite{F} obtained a classification of effective, local circle actions on closed topological $3$-manifolds, generalizing the classification of closed topological $3$-manifolds with effective circle actions in \cite{R} and \cite{OR}:


\begin{theorem}[Orlik and Raymond \cite{OR2}, Fintushel \cite{F}]
\label{THM:ORF}
A closed, connected topological $3$-manifold $M$ with an  effective local $S^1$-action is determined up to equivariant equivalence by a set of fiber invariants 
\[
\{b; \varepsilon, g, (f,k_1), (t,k_2); \{ (\alpha_i, \beta_i) \}_{i=1}^n \}.
\]
\end{theorem}
We recall the definition of these invariants in Section~\ref{S:Orbit_types}.
The topological classification of the manifolds in Theorem~\ref{THM:ORF} is given in \cite[Theorem in p.~143]{OR2} and \cite[Section~3]{F}.

We consider effective and isometric local circle actions on closed Alexandrov $3$-spaces, i.e effective local circle actions such that the $S^1$-actions on the tubular neighborhoods of the curves of the decomposition are isometric with respect to the restricted metric. This class of spaces contains the class of closed Alexandrov $3$-spaces admitting an isometric circle action as well as the class of closed Seifert $3$-manifolds.

Recall that the space of directions $\Sigma_xX$ of a point $x$ in an Alexandrov $3$-space $X$ without boundary must be homeomorphic either to a $2$-sphere or to a real projective plane $\mathbb{R}P^2$ (see Section~\ref{S:PRELIM}). We say that $x\in X$ is \emph{topologically singular} if $\Sigma_xX$ is homeomorphic to $\mathbb{R}P^2$.
Let $\Susp(\mathbb{R}P^2)$ denote the suspension of the real projective plane. Our main result, which generalizes \cite[Theorem 1.2]{NZ}, is the following:


\begin{theorem}
\label{THM:INVARIANTS}
Let $X$ be a closed, connected Alexandrov $3$-space admitting an isometric local $S^1$-action. Assume that $X$ has $2r\geq 0$ topologically singular points. Then the following hold:
\begin{itemize}
\item[1.] The set of isometric local circle actions	 (up to equivariant equivalence)  is in one-to-one correspondence with the set of unordered tuples
\[
\left\{b; \varepsilon, g, (f,k_1), (t,k_2), (s,k_3); \{ (\alpha_i, \beta_i) \}_{i=1}^n; (r_1,r_2, \ldots, r_{s-k_{3}}); (q_1, q_2, \ldots, q_{k_3})\right\}.
\]
where the permissible values for $b$, $\varepsilon$, $g$, $(f,k_1)$, $(t,k_2)$ and $(\alpha_i,\beta_i)$ are given by Theorem \ref{THM:ORF}, and 
$(r_1, r_2, \ldots, r_{s-k_{3}})$ and $(q_1, q_2, \ldots, q_{k_3})$  are unordered $(s-k_3)$- and $k_3$-tuples of non-negative even integers $r_i$, $q_j$, respectively, such that $r_1+\ldots + r_{s-k_{3}} + q_1 +\ldots q_{k_3}=2r$.\\ 

\item[2.] The space $X$ is equivariantly equivalent to 
\[
M\#\underbrace{\Susp(\Real P^2)\# \cdots \# \Susp(\Real P^2)}_{r \text{ summands}},
\]
where $M$ is the closed $3$-manifold determined by the set of invariants 
\[
\left\{ b; \varepsilon, g, (f+s,k_1+k_3), (t,k_2); \{ (\alpha_i, \beta_i) \}_{i=1}^n \right\}
\]
in Theorem \ref{THM:ORF}.
\end{itemize}
\end{theorem}

We remark that the set of closed Alexandrov $3$-spaces admitting an effective and isometric local circle action is contained in the set of collapsing Alexandrov $3$-spaces considered by Mitsuishi and Yamaguchi in \cite{MY}. In our case, there are no singular fibers having neighborhoods of the form $B(pt)$ (see \cite[Definition 2.48]{MY}). The structure induced by the local action allows us to give, via Theorem~\ref{THM:INVARIANTS}, an alternative description to the one in \cite{MY} for those spaces without $B(pt)$ building blocks, exhibiting them as one of the connected sums in the theorem. Additionally, we conclude that whenever a sequence of closed Alexandrov $3$-spaces collapses to an Alexandrov surface (with or without boundary) without singular fibers of type $B(pt)$, this collapse is given by an effective and isometric local circle action (see Corollary~\ref{C:COLLAPSE}).  
This is reminiscent of the fact that Riemannian manifolds that collapse with  bounded sectional curvature admit a so-called \emph{$F$-structure}, which is, roughly speaking, a generalized local torus action (see \cite{ChGr1,ChGr2}). 
\\

The contents of the present article are organized as follows. In Section \ref{S:PRELIM} we recall some basic results on isometric Lie group actions on Alexandrov spaces. In Section \ref{S:Orbit_types} we describe the structure of the fiber space of a closed Alexandrov $3$-space $X$ with an effective, isometric local circle action. We also assign equivariant and topological invariants to $X$.  In Section~\ref{S:CONSTRUCTIONS} we show how to construct closed Alexandrov $3$-spaces with prescribed local circle actions out of a given system of invariants. Section \ref{S:CLASS_GENERAL} contains the proof of Theorem \ref{THM:INVARIANTS}. Finally, in Section \ref{S:COLLAPSE} we relate our results to those in the collapsing theory of Alexandrov $3$-spaces.


\begin{ack} The authors would like to thank Luis Guijarro for helpful conversations. 
\end{ack}


\section{Preliminaries}
\label{S:PRELIM}

In this section we recall some basic facts on isometric actions of compact Lie groups on finite-dimensional Alexandrov spaces. We refer the reader to \cite{Br} for a thorough exposition on the general theory of compact transformation groups. The elements of Alexandrov geometry can be found in \cite{BBI} and \cite{BGP}. 


\subsection*{Group actions}
Let $X$ be a finite-dimensional Alexandrov space.
As in the Riemannian case, the isometry group $\mathrm{Isom}(X)$ of $X$ is a Lie group (see \cite{FY}) and is compact (in the compact-open topology) whenever $X$ is compact. Let $G\times X\to X$ be an isometric action of a compact Lie group $G$ on $X$. The \emph{orbit} of a point $x\in X$ is the set $G(x)=\{ gx\in X \mid g\in G\}$; the \emph{isotropy group} at $x$ is the closed subgroup $G_x=\{g\in G \mid gx=x \}$. There is a natural homeomorphism $G(x)\cong G/G_x$, for each $x\in X$. The closed subgroup $\bigcap_{x\in X}G_x$ of $G$ is the \textit{ineffective kernel} of the action. If the ineffective kernel is trivial, we will say that the action is \textit{effective}. In what follows we will only consider effective actions.  

Given a subset $A\subseteq X$, we denote its image with respect to the orbit projection map $\mathsf{p}: X\to X/G$ by $A^*$. In particular, $X^{*}=X/G$ is the orbit space of the action. It was proved in \cite[Section 4.6]{BGP} that the orbit space $X^*$ is an Alexandrov space with the same lower curvature bound as $X$. Given a subset $A$ of the space of directions $\Sigma_xX$ of $X$ at $x$, the \textit{set of normal directions to $A$} is, by definition, the set
\[
A^{\perp}:=\{ v\in \Sigma_xX \mid d(v,w)=\mathrm{diam}(\Sigma_xX)/2 \ \  \text{for all} \ \ w\in A  \}.
\]
The following proposition describes the tangent and normal spaces to the orbits (cf. \cite[Proposition 4]{GS}). 


\begin{prop}
Let $X$ be an Alexandrov space admitting an isometric $G$-action and fix $x\in X$ with $\dim G/G_x>0$. If $S_x\subset \Sigma_xX$ is the unit tangent space to the orbit $G(x)\cong G/G_x$, then the following hold:
\begin{itemize}
\item[(1)] The set $S_x^{\perp}$ is a compact, totally geodesic Alexandrov subspace of $\Sigma_xX$ with curvature bounded below by $1$, and the space of directions $\Sigma_xX$ is isometric to the join $S_x\star S_x^{\perp}$ with the standard join metric. 

\item[(2)] Either $S_x^{\perp}$ is connected or it contains exactly two points at distance $\pi$.
\end{itemize}
\end{prop}

The (Euclidean) cone of an Alexandrov space $Y$ of $\mathrm{curv}\geq 1$ is denoted by $K(Y)$ and it is assumed to have the standard Euclidean cone metric.

Let $X$ be a three-dimensional Alexandrov space without boundary and an isometric circle action. Because of the low codimension of the orbits, the Slice Theorem (cf.~\cite[Theorem~B]{HS}) holds  by purely topological reasons (see \cite{F} and \cite[Remark 4.6]{NZ}) and a slice at $x$ is equivariantly homeomorphic to $K(S^{\bot}_x)$. It follows that  $\Sigma_{x*}X^*$, the space of directions at $x^*$ in $X^*$, is isometric to $S^{\bot}_x/G_x$. We will also use the Alexandrov versions of the isotropy Lemma (\cite[Lemma 2.1]{GG}) and the Principal Orbit Theorem (\cite[Theorem 2.2]{GG}). 

Let $G$ act isometrically on two Alexandrov spaces $X$ and $Y$. A mapping $\varphi: X\rightarrow Y$ is \textit{weakly $G$-equivariant} if, for every $x\in X$ and $g\in G$, there exists an automorphism $f$ of $G$ such that $\varphi(gx)= f(g)\varphi(x)$; if $f$ is the identity homomorphism, then we simply say that $\varphi$ is \textit{$G$-equivariant}. If it is clear which group $G$ is, we will only say that $\varphi$ is \textit{weakly equivariant} or \textit{equivariant}, depending on the situation. Two actions of $G$ on $X$ are said to be \textit{equivalent} if there exists a weakly equivariant homeomorphism from $X$ onto itself.


\subsection*{Three-dimensional Alexandrov spaces}
 We briefly recall the basic \linebreak structure of closed Alexandrov spaces of dimension $3$. For more details on these spaces, we refer the reader to \cite{GG1}.
 
Let $X$ be a closed Alexandrov $3$-space.  Recall that the space of directions at each point in $X$ is a closed $2$-dimensional Alexandrov space with curvature bounded below by $1$. Therefore, the Bonnet--Myers Theorem  (see \cite[Theorem 10.4.1]{BBI}) implies that $\Sigma_xX$ has finite fundamental group, and is therefore homeomorphic to either a $2$-sphere $\mathbb{S}^2$ or a real projective plane $\mathbb{R}P^2$. The set of points having space of directions homeomorphic to $\mathbb{S}^2$ is open and dense in $X$. We call a point in $X$ \emph{topologically regular} if its space of directions is homeomorphic to $\sphere^2$ and \emph{topologically singular} if its space of directions is homeomorphic to $\Real P^2$.

The Conical Neighborhood Theorem of Perelman \cite[I.\ Local Theorem]{Per2} states that every point $x$ in an Alexandrov space has a neighborhood pointed-homeomorphic to the cone over the space of directions at $x$. As a consequence of this result, the set of points in $X$ with space of directions $\mathbb{R}P^2$ is finite and $X$ must be homeomorphic to a compact $3$-manifold with finitely many $\mathbb{R}P^2$-boundary components to which one glues in cones over $\mathbb{R}P^2$. It is not difficult to see that $X$ must have an even number of topologically singular points.

A closed non-manifold Alexandrov $3$-space $X$ can also be described as a quotient of a closed, orientable, topological $3$-manifold $\tilde{X}$ by an orientation-reversing involution $\iota:\tilde{X}\to \tilde{X}$ with only fixed points. The manifold $\tilde{X}$ is the \textit{ orientable double branched cover of $X$} (see, for example, \cite[Lemma 1.7]{GG1}).  Furthermore, the metric on $X$ can be lifted to $\tilde{X}$ so that $\tilde{X}$ is an Alexandrov space with the same curvature bound as $X$  and $\iota:\tilde{X}\to \tilde{X}$ is an isometry with respect to the lifted metric; in particular, $\iota$ is equivalent to a smooth involution on the $3$-manifold $\tilde{X}$  (see \cite[Lemma 1.8]{GG1} and \cite[Section 5]{GW} for details).


\section{Local circle actions}
\label{S:Orbit_types}

Let $X$ be a closed, connected Alexandrov $3$-space. We say that $X$ admits an \emph{isometric local $S^1$-action} if it can be decomposed into (possibly degenerate) disjoint, simple, closed curves each having a tubular neighborhood which admits an effective, isometric  $S^1$-action (with respect to the  restricted metric) whose orbits are the curves of the decomposition. Here, by \textit{a tubular neighborhood} of a subset $A\subseteq X$, we mean an $\varepsilon$-neighborhood of $A$ for some small $\varepsilon>0$.   We call each element of the decomposition into curves of $X$ a \emph{fiber} and $X^{*}$ the \emph{fiber space} of the local $S^1$-action.  The \emph{fiber map} $\mathsf{p}:X\rightarrow X^*$ is defined as the map that coincides with the orbit maps of the circle actions on each of the neighborhoods of the decomposition of the local $S^1$-action. 
We point out that, as opposed to the manifold case, a sufficiently small tubular neighborhood of a fiber may not be homeomorphic to a disk-bundle over the fiber. This is the case, for example, when the fiber is a topologically singular point  $x\in X$ (corresponding to a so-called \emph{$SF$-fiber}, defined below): a small tubular neighborhood of $x$ is a cone over $\RP^2$ and corresponds to the preimage (under the fiber map) of a small neighborhood of $x^*$ in the fiber space. More generally, it follows from the definition of a local circle action that a small tubular neighborhood of a fiber is a union of curves which are elements of the decomposition of $X$. In other words, a small tubular neighborhood of a fiber is the preimage  of  a neighborhood of a point in the fiber space.

Following Orlik, Raymond \cite{OR2}, and Fintushel \cite{F}, we will study the structure of the fiber space, the types of fibers that can arise and compare them to the manifold case.


\subsection*{Fiber types.}
We now describe the possible fibers of the decomposition of $X$ induced by the local circle action. The fiber types consisting of topologically regular points were first considered in \cite{OR2,F}. We first list the possible degenerate fibers of the local circle action, i.e.\ those fibers consisting of just one point. 
These fibers correspond to the fixed points of the local circle action.
\\


\noindent\emph{$F$-fibers.} Topologically regular single-point fibers will be called \emph{$F$-fibers}. By \cite{MSY} a sufficiently small invariant metric ball around an $F$-fiber is equivariantly homeomorphic to an orthogonal action on a $3$-ball. 
We denote the set of $F$-fibers by $F$.
\\


\noindent\emph{$SF$-fibers.} Let $V$ be a small invariant tubular neighborhood of a fiber of the local circle action on $X$. Suppose that $V$ contains a topologically singular point $p$ of $X$. Since the $S^1$-action on $V$ is isometric, $p$ is a fixed point. In other words, topologically singular points of $X$ are degenerate fibers of the decomposition of $X$ into curves. A sufficiently small invariant metric ball $B_r(p)$ around $p$ is homeomorphic to a cone over $\Real P^2$. Therefore, by \cite[Corollary 4.5]{NZ}, the restriction of any local $S^1$-action on $X$ to $B_r(p)$ is equivalent to the cone of the standard cohomogeneity one circle action on the unit round $\Real P^2$. 
We describe this action explicitly. 

We let each point in $\Real P^2$ be an equivalence class $[r,\theta]$, where $(r,\theta)$ lies on the unit disk with polar coordinates $0\leq r \leq 1$, $0\leq \theta \leq 2\pi$, and $(1,\theta)$ is identified with $(1, \theta + \pi)$ for all $\theta$. Therefore, the points of the cone over $\Real P^2$ are equivalence classes $[[r,\theta], t]$,  where $[r,\theta]\in \Real P^2$, $0\leq t\leq 1$, and the points of the form $([r,\theta],0)$ represent a single point. We let $K(\Real P^2)$ be the cone over $\Real P^2$. Then, for every $0\leq \varphi\leq 2\pi$, the cone of the standard cohomogeneity one circle action on the unit round $\Real P^2$ is given by 
\begin{eqnarray*}
S^1 \times K(\Real P^2) & \longrightarrow & K(\Real P^2) \\
(\varphi, [[r,\theta], t]) & \longmapsto & [[r,\theta + \varphi], t].
\end{eqnarray*}

We will call topologically singular fibers \emph{$SF$-fibers}. Observe that $SF$-fibers are isolated in $X$. 
We denote the stratum of $SF$-fibers by $SF$.
\\

We now describe the possible non-degenerate fibers. Observe that the notion of local orientability makes sense at topologically regular points of $X$. 
\\


\noindent\emph{$E$-fibers.} Let $\mu$ and $\nu$ be two relatively prime integers satisfying $0<\nu<\mu$. Consider the following $S^1$ action over a solid torus:
\begin{eqnarray*}
S^1 \times \left(D^2\times \mathbb{S}^1\right) & \longrightarrow & \left(D^2\times \mathbb{S}^1\right) \\
(z, \rho e^{i\theta}, e^{i\psi}) & \longmapsto & (z^{\nu}\rho e^{i\theta}, z^{\mu}e^{i\psi}).
\end{eqnarray*}
 The curve $\rho = 0$ will be called an \emph{exceptional fiber} or \emph{$E$-fiber}. Tubular neighborhoods of $E$-fibers are of the form $D^2\times_{\mathbb{Z}_\mu}\sphere^1$ where $\mathbb{Z}_{\mu}$ acts on $D^2$ by rotations without reversing the local orientation and with the $E$-fiber corresponding to the curve $\{0\}\times_{\mathbb{Z}_\mu}\sphere^1$. We assign Seifert invariants $(\alpha, \beta)$ to the $E$-fiber as in \cite{OR} (see also \cite{O}). 
 We denote the stratum of $E$-fibers by $E$.
 \\
 
 
\noindent\emph{$SE$-fibers.} We consider now the case where $\mathbb{Z}_2$ acts by reflexion with respect to an axis of $D^2$, reversing the local orientation. To describe the action, we let $J=(-1,1)$, $I$ be an open interval and identify $D^2$ with $J\times I$. Then we have the following $\mathbb{Z}_2$ action on $D^2\times \sphere^1$:
\begin{eqnarray*}
\mathbb{Z}_2 \times \left(J\times I\times \mathbb{S}^1\right) & \longrightarrow & \left(J\times I\times \mathbb{S}^1\right) \\
(t, u, z) & \longmapsto & (-t, u, -z).
\end{eqnarray*}
The quotient $D^2\times_{\mathbb{Z}_2}\sphere^1$ of the previous action is homeomorphic to $\mathrm{Mo}\times I$, where $\mathrm{Mo}$ is the M\"{o}bius band $(J\times S^1)/(t,z)\!\sim\! (-t,-z)$. We denote the elements of  $D^2\times_{\mathbb{Z}_2}\sphere^1$ by $[t,u,z]$. Now, we consider the following circle action: 
\begin{eqnarray*}
S^1 \times (D^2\times_{\mathbb{Z}_2} \sphere^1) & \longrightarrow & D^2\times_{\mathbb{Z}_2} \sphere^1 \\
(w,[t,u,z]) & \longmapsto & [t,u,wz].
\end{eqnarray*}  

 The fibers of the previous action intersecting $(\{0\}\times I)\times_{\mathbb{Z}_2}\sphere^1$ 
 will be called \emph{special exceptional fibers} or \emph{$SE$-fibers}. We denote the stratum of $SE$-fibers by $SE$.
 \\
 
 
 \noindent\emph{$R$-fibers.} Fibers that are not $F$-, $SF$-, $E$- or $SE$-fibers will be called \emph{$R$-fibers}. 
 The restriction of the fiber map $\mathsf{p}$ to the stratum of $R$-fibers is an $S^1$ fiber bundle with structure group $\mathrm{O}(2)$.  
 We denote the stratum of $R$-fibers by $R$.
 
 
\subsection*{Fiber space structure.} Observe first that, since the fiber projection map $\mathsf{p}:X\to X^*$ is a local submetry,  $X^*$ satisfies the triangle comparison condition locally and is therefore a $2$-dimensional Alexandrov domain (see \cite[Corollary in p.~16]{BGP}). Therefore, $X^*$ is a topological $2$-manifold (see \cite[Corollary 10.10.3]{BBI}) and its boundary is composed of the images of $F$-, $SF$- and $SE$-fibers under the fiber map, while the interior of $X^*$ consists of $R$-fibers and a finite number of $E$-fibers. This structure for $X^*$ also follows, using  purely topological arguments, from 
\cite[Lemma~1]{R} (see also \cite[Section~1.9, Lemma~1]{O}) and \cite[Proposition 3.2]{NZ}.


\subsection*{Block types.}
As shown in  \cite[Section 2]{F} and  \cite[Section 1]{OR}, a closed $3$-manifold with a local circle action can be decomposed into different kinds of building blocks;  these arise when considering tubular neighborhoods of connected components of fibers of the same type.
We now recall the definition of these blocks in the manifold case and define new types of building blocks to account for the presence of topologically singular points in a non-manifold three-dimensional Alexandrov space.


\subsubsection*{Manifold blocks.} We first list the blocks that contain only topologically regular points; these arise when considering local circle actions on $3$-manifolds.
\\
 
\noindent\emph{$E$-blocks.} We will call a small tubular neighborhood of an exceptional fiber with Seifert invariants $(\alpha,\beta)$ an \emph{$E$-block of type $(\alpha,\beta)$}. The orbit space of an $E$-block is a $2$-disk where a single point on the interior corresponds to the exceptional fiber and the rest of the fibers have trivial isotropy. 
\\


\noindent\emph{Simple and twisted $F$-blocks.} Let $C_F$ be a component of $F$-fibers and let $V$ be a sufficiently small tubular neighborhood of $C_F$. Since $\partial V$ is composed of $R$-fibers, we have the following possibilities for $\pi|_{V}:V\rightarrow V^{*}$. If $\partial V$ is orientable, then it is the trivial bundle $\sphere^1\times \sphere^1$ and $V$ is a solid torus equipped with the following circle action:
\begin{eqnarray*}
S^1 \times (D^2\times \sphere^1) & \longrightarrow &  D^2\times \sphere^1 \\
(z,\rho e^{i\theta}, e^{i\psi}) & \longmapsto & (z\rho e^{i\theta}, e^{i\psi} ).
\end{eqnarray*} 
Here, the curve $\rho=0$ is the $F$-component. We call $V$ a \emph{simple $F$-block}.

If $\partial V$ is non-orientable, then it is the non-orientable bundle $\sphere^1\tilde{\times}\sphere^1$, that is, a Klein bottle $K$. Therefore $V$ is homeomorphic  to a solid Klein bottle $K\times [0,1]$ with the fibers contained in $K\times \{0\}$ collapsed to points, which correspond to $C_F$. 
 We will call this block $V$ a  \emph{twisted $F$-block}.
 Any local circle action on a twisted $F$-block coincides locally, around a point on $C_F$, with the circle action described for simple $F$-blocks. However, that circle action cannot be extended to a global circle action on a twisted $F$-block, since there is no coherent choice of orientations for its fibers. In other words, the structure group of the bundle $\partial V$ does not reduce to $\SO(2)$ on a twisted $F$-block.  
\\


\noindent\emph{Simple and twisted $SE$-blocks.} Let $C_{SE}$ be a component of $SE$-fibers and let $U$ be a sufficiently small tubular neighborhood of $C_{SE}$. As in the case of $C_F$ we look at the restricted bundle $\pi|: U \rightarrow U^{*}$. If $\partial U$ is orientable, then $U$ is equivariantly homeomorphic to $\mathrm{Mo}\times\mathbb{S}^1$, the product of the M\"{o}bius band and a circle (see \cite[Section 1]{OR2} and  \cite[Section 1.8]{O}). We will call $U$ a \emph{simple $SE$-block.}

If $\partial U$ is non-orientable, then the block is $\mathrm{Mo}\tilde{\times}\mathbb{S}^1$, the non-trivial $\mathrm{Mo}$-bundle over $\mathbb{S}^1$. Explicitly, the block can be described as $[0,1]\times\mathbb{S}^1\times \mathbb{S}^1$ after identifying $(0,e^{i\theta},e^{i\psi})\sim(0,e^{i\theta},e^{i\psi+\pi})$ and then taking the image of each $\mathbb{S}^1\times \mathbb{S}^1\times\{t\}$ under the usual covering of the Klein bottle by the torus. 
We will call $U$ a \emph{twisted $SE$-block.}


\subsubsection*{Non-manifold blocks.} In addition to the manifold blocks described in \cite{OR} and \cite{F}, we define two more types of building blocks. These contain topologically singular points and arise when considering local circle actions on non-manifold Alexandrov $3$-spaces.
\\


\noindent\emph{Simple $SF$-blocks.} Let $(\mathrm{Susp}(\Real P^2), d_0)$ be the suspension of the real projective plane (with constant curvature $1$) equipped with the standard spherical suspension Alexandrov metric of curvature bounded below by $1$. Consider the space $(\mathrm{Susp}(\Real P^2), d_0)$ with the suspension of the standard circle action on $\Real P^2$.
 We let $R_k$ be the \emph{equivariant connected sum} (as in  \cite[Section 4.1]{NZ}) of $k$ copies of $\mathrm{Susp}(\Real P^2)$. By  \cite[Corollary 4.5]{NZ} there is a unique (up to equivalence) isometric, effective circle action on $R_k$. We define the \emph{simple $SF$-block} as $R_k\setminus (D^2\times \mathbb{S}^1)$ equipped with the restricted circle action, and with the invariant solid torus removed from the principal part of $R_k$. The orbit space of a simple $SF$-block is an annulus with the following structure: the interior and one of the boundary components of the orbit space are composed of principal orbits; the remaining boundary component is made up of a union of arcs joined by their endpoints, alternating between $S^1$ and $\mathbb{Z}_2$ isotropy groups. 
\\


\noindent\emph{Twisted $SF$-blocks.} Let $V$ be a twisted $F$-block and  let $R_k$ be the equivariant connected sum of $k$ copies of $\Susp(\mathbb{R}P^2)$ (as in the construction of the simple $SF$-block). We consider $F$-fibers $x$ in $V$ and $y$ in $R_k$, and let $B_x$ and $B_y$ be small invariant open neighborhoods of $x$ and $y$, respectively. We may assume that $B_x$ and $B_y$ are small enough so that their closures contain only topologically regular fixed points and points with trivial isotropy.  By the description of the local circle action on $V$ and the global circle action on $R_k$, we can further assume that $\overline{B}_x$ and $\overline{B}_y$, the closures of $B_x$ and $B_y$, respectively,  are equivariantly homeomorphic to closed $3$-balls with effective and isometric circle actions. We define a \textit{twisted $SF$-block} as $V\# R_k$, the equivariant connected sum of $V$ and $R_k$ along $B_x$ and $B_y$ (as in  \cite[Section 4.1]{NZ}).


\subsection*{Fiber space invariants.} Besides the information given by the local action, we must consider the topological type of the fiber space $X^*$ as well. We will list invariants associated to the topological type of $X^{*}$ and the local $S^1$-action. To this end, we  recall the classification up to weak bundle equivalence of $S^1$ bundles with structure group $\mathrm{O}(2)$ over a compact $2$-manifold with boundary (see \cite[Section~1]{F}, also \cite{O,OR2,OVZ}). 


\begin{definition}Let $\pi_i: E_i\to B_i$, $i=1,2$, be $S^1$ fiber bundles with structure group $\mathrm{O}(2)$. A \emph{weak equivalence} between these bundles is a homeomorphism $\varphi:E_1\to E_2$ that covers a homeomorphism $f:B_1\to B_2$, i.e.~the diagram
\[
\xymatrix{
E_1 \ar[r]^{\varphi}\ar[d]_{\pi_1}& E_2\ar[d]^{\pi_{2}} \\
B_1  \ar[r]^f&  B_2
}
\]
commutes. If the total spaces $E_1,E_2$ are orientable, we require that an orientation be chosen for $E_1$ and $E_2$ and that $\varphi$ preserves orientation. If either of the structure groups reduces to $\SO(2)$ we demand that $\varphi$ be a bundle map with respect to $\SO(2)$. Two $S^1$ fiber bundles with structure group $\mathrm{O}(2)$ are \emph{weakly equivalent} if there exists a weak equivalence between them. 
\end{definition}

Recall that the fundamental group of a compact, connected $2$-manifold $B$ with genus $g$ and $m>0$ boundary components has the presentation 
\[
( a_j, b_j, s_i \ | \ s_1\cdots s_m [a_1,b_1]\cdots[a_g,b_g] )
\] 
if $B$ is orientable, and 
\[
(v_j, s_i \ | \ s_1\cdots s_mv_1^2\cdots v_g^2)\]
 if $B$ is nonorientable.


\begin{thm}[\protect{cf.~\cite[Theorem~1]{F}}]
\label{THM:CLASS_BUNDLES}
Let $B$ be a compact, connected $2$-manifold with $m>0$ boundary components and genus $g$. Then the set of weak equivalence classes of circle bundles over $B$ with structure group $\mathrm{O}(2)$ is in one-to-one correspondence with the pairs $(\varepsilon, k)$ where $k\geq 0$ is an even integer corresponding to the number of $s_i$ in the presentation of $\pi_1(B)$ that reverse orientation along fibers. The symbol $\varepsilon$ can take the values $o_1, o_2, n_1, n_2, n_3, n_4$ representing the following classes:
\begin{itemize}
\item[$o_1$:] $B$ is orientable and all $a_j$, $b_j$ preserve orientation.
\item[$o_2$:] $B$ is orientable, all $a_j$, $b_j$ reverse orientation and $g\geq 1$. 
\item[$n_1$:] $B$ is non-orientable, all $v_j$ preserve orientation and $g\geq 1$. 
\item[$n_2$:] $B$ is non-orientable, all $v_j$ reverse orientation and $g\geq 1$.
\item[$n_3$:] $B$ is non-orientable, $v_1$ preserves orientation, all other $v_j$ reverse orientation and $g\geq 2$.
\item[$n_4$:] $B$ is non-orientable, $v_1$ and $v_2$ preserve orientation, all other $v_j$ reverse orientation and $g\geq 3$. 
\end{itemize} 
If $k>0$, then the classes $o_1$ and $o_2$ collapse to a single class $o$, and the classes $n_1$, $n_2$, $n_3$ and $n_4$ collapse to a single class $n$.
\end{thm}

We now associate the following invariants to a fiber space $X^*$. Let $(\varepsilon,k)$ be the pair associated to the bundle of $R$-fibers with possible values as in Theorem \ref{THM:CLASS_BUNDLES}. We denote the genus of $X^{*}$ by $g\geq 0$.  We let $f, t, k_1, k_2$ be non-negative integers  such that $k_1 \leq f$ and $k_2\leq t$, where $k_1$ is the number of twisted $F$-blocks and $k_2$ is the number of twisted $SE$-blocks. Consequently $f-k_1$ is the number of simple $F$-blocks and $t-k_2$ is the number of simple $SE$-blocks. A non-negative integer $n$ will denote the number of $E$-fibers and we let $\{ (\alpha_i, \beta_i)\}_{i=1}^n$ be the corresponding Seifert invariants. We also let $b$ denote an integer or an integer mod $2$ with the following conditions: $b=0$ if $f+t>0$ or if $\varepsilon\in\{o_2,n_1,n_3,n_4 \}$ and some $\alpha_i=2$; $b\in\{0,1\}$ if $f+t=0$ and $\varepsilon\in\{ o_2,n_1,n_3,n_4\}$ and all $\alpha_i\neq 2$. In the remaining cases $b$ is an arbitrary integer. 
We let $s, k_3$ be non-negative integers, where  $k_3\leq s$ is the number of twisted $SF$-blocks. Hence $s-k_3$ is the number of simple $SF$-blocks, and we let $(r_1, r_2, \ldots, r_{s-k_{3}})$ and $(q_1, q_2, \ldots, q_{k_3} )$ be $(s-k_{3})$- and $k_3$-tuples of non-negative even integers corresponding to the number of topologically singular points in each  simple and twisted $SF$-block, respectively. The numbers $k$, $k_1$, $k_2$, and $k_3$ satisfy $k_1 + k_2 + k_3 = k$. Summarizing, to any fiber space $X^{*}$ we associate the set of invariants
\[
\left\{b; \varepsilon, g, (f,k_1), (t,k_2), (s,k_3); \{ (\alpha_i, \beta_i) \}_{i=1}^n; (r_1,r_2, \ldots, r_{s-k_{3}}); (q_1, q_2, \ldots, q_{k_3})\right\}.
\]

\begin{definition}[cf.~\protect{\cite[p.~150]{OR2} and \cite[p.~116]{F}}]
Let $X$ and $Y$ be two closed, connected Alexandrov $3$-spaces admitting local isometric $S^1$ actions. We will say that their fiber spaces are \emph{isomorphic} if there is a weight-preserving homeomorphism $X^{*} \rightarrow Y^{*}$, i.e.~a homeomorphism that preserves the fiber space invariants. In the case that $X^{*}$ and $Y^{*}$ are oriented we require the homeomorphism to be orientation-preserving. 
We will say that $X$ and $Y$ are \emph{equivariantly equivalent} if there is a fiber-preserving homeomorphism $X\rightarrow Y$ which is orientation preserving on $X\setminus (SE \cup SF)$ when $X\setminus (SE \cup SF)$ is oriented. 
In the case of a global circle action, one can show that equivariant equivalence reduces to equivariant homeomorphism (cf.~ \cite[p.~150]{OR2}). 
\end{definition}



\section{Constructing spaces with local circle actions}
\label{S:CONSTRUCTIONS}
In this section we show how to construct a closed three-dimensional Alexandrov space with an effective and isometric local $S^1$-action out of the set of invariants
\begin{align}
\label{EQ:INVARIANTS}
\left\{b; \varepsilon, g, (f,k_1), (t,k_2), (s,k_3); \{ (\alpha_i, \beta_i) \}_{i=1}^n; (r_1,r_2, \ldots, r_{s-k_{3}}); (q_1, q_2, \ldots, q_{k_3})\right\}.
\end{align}
defined in the previous section. These invariants determine a topological space with a topological effective local circle action in the following manner.

 If $f+t>0$ we let $X_0^{*}$ be a $2$-manifold of genus $g$ and $f+t+n+s$ boundary components which is orientable if $\varepsilon\in\{o,o_1,o_2\}$ and non-orientable if $\varepsilon \in \{n, n_1,n_2,n_3,n_4\}$. Let $X_0$ be the circle bundle with structure group $\mathrm{O}(2)$ over $X_0^{*}$ associated to $(\varepsilon, k)$,  with $k=k_1 + k_2 + k_3$. This bundle has a cross-section $q$ (see  \cite[Section 2]{OR2}). On $n$ of the torus boundary components $q$ restricts to curves $q_i$ and the structure group on these tori reduces to $\mathrm{SO}(2)$. This determines equivariant sewings for the $E$-blocks onto the $n$ torus components as in the proof of  \cite[Theorem 1.10]{O}. To the remaining boundary components we attach $k_1$  twisted $F$-blocks, $k_2$ twisted $SE$-blocks, $f-k_1$ simple $F$-blocks and $t-k_2$ simple  $SE$-blocks. Similarly, by means of fiber-preserving homeomorphisms we glue $s-k_3$ simple $SF$-blocks (the $j$-th block having $r_j$ topologically singular points) and $k_3$ twisted $SF$-blocks (where the $j$-th block has $q_j$ topologically  singular points). Then we let $X$ be the space obtained by this procedure.

If $f=t=0$, we let $M$ be the $3$-manifold determined by the set of invariants $\left\{b; \varepsilon, g, (s, k_3), (0,0); \{ (\alpha_i, \beta_i) \}_{i=1}^{n} \right\}$ as in \cite[Theorem 0]{OR2} and \cite[Theorem 2]{F}.  If $s=0$, then $M$ itself is the space we consider. Let us recall that the equivariant connected sum (see~\cite[Section 4.1]{NZ}) of $j$ copies of $\Susp(\mathbb{R}P^2)$ will be denoted by $R_j$ as in the construction of the simple $SF$-blocks in Section \ref{S:Orbit_types}. If $s>0$, then, for each $i=1,\ldots, s-k_3$, we take an equivariant connected sum of $R_{r_i/2}$ with $M$ centered at an $F$-fiber of $M$ lying on an $F$-component belonging to a simple $F$-block of $M$. We denote the space obtained by $Y$. Finally, for every $j=1,\ldots, k_3$, we perform an equivariant connected sum of $R_{q_j/2}$ with $Y$ centered at an $F$-fiber of $Y$ lying on an $F$-component belonging to a twisted $F$-block of $Y$. We let $X$ be the resulting space. 

The space $X$ constructed in the preceding paragraphs is only a topological space with an effective topological local circle action and, a priori, $X$ is not  an Alexandrov space. We will now show that it is possible to equip $X$ with an Alexandrov metric such that the local circle action is isometric. We distinguish two cases, depending on whether or not $X$ is a topological manifold. In the case where $X$ is not a topological manifold,  we observe that $X$ is 
 homeomorphic to a smooth
 orbifold since it is a quotient of its double branched cover, a closed  topological $3$-manifold, by an involution that is equivalent to a smooth involution (see Section~\ref{S:PRELIM}). We will first  prove that $X$ admits a Riemannian metric (or an orbifold Riemannian metric in the case that $X$ is not a manifold) such that the action is isometric (with respect to the Riemannian metric). Then we prove that the local circle action is isometric with respect to the distance induced by this metric. As a first step, we will show that, in the case where $X$ is not a manifold, the topological local circle action on $X$ lifts to a topological local circle on the orientable double branched cover of $X$ (see Proposition~\ref{T:DBL_BRANCHED_COVER} below). Then we prove that $X$ admits an invariant (orbifold) Riemannian metric $g$ (see Proposition~\ref{L:RIEMANNIAN_METRIC}). Finally, we show that the local circle action is isometric with respect to the distance function on $X$ induced by the (orbifold) Riemannian metric $g$ (see Proposition~\ref{L:METRIC}).
 
 
 Let $X$ be the topological space constructed out of the set of invariants given in \eqref{EQ:INVARIANTS} and assume that $X$ is not a topological manifold. As recalled in the preceding paragraph, there exists a closed orientable topological $3$-manifold $\tilde{X}$ with an involution $\iota:\tilde{X}\to \tilde{X}$ such that $X$ is homeomorphic to $\tilde{X}/\iota$ and $\iota$ is equivalent to a smooth involution on $\tilde{X}$. The manifold $\tilde{X}$ is the \emph{orientable  double branched cover} of $X$.


\begin{prop}
\label{T:DBL_BRANCHED_COVER}
Let $X$ be the topological space with a topological, effective local circle action constructed out of the set of invariants given in \eqref{EQ:INVARIANTS}. If $X$ is not a manifold, then the topological local circle action on $X$ lifts to an effective local circle action on the orientable double branched cover $\tilde{X}$.
\end{prop} 


\begin{proof}
Let $\Gamma= \{\gamma\}_{\gamma\in\Lambda}$  be the decomposition of $X$ into fibers of the local circle action. We let $\pi: \tilde{X}\to X$ be the canonical projection. Let $X_0$ be the space resulting from taking out small conical neighborhoods of each topologically singular point of $X$ and let $\tilde{X}_0$ be the orientable double cover of $X_0$.
For each $\gamma\in \Gamma$ we consider $V_{\gamma}$ a small tubular neighborhood of $\gamma$. Then we have an effective $S^1$-action
\[
\mu: S^1\times V_{\gamma}\to V_{\gamma}.
\]	
Consider $\pi|_{\pi^{-1}(V_{\gamma})}:\pi^{-1}(V_{\gamma})\to V_{\gamma}$, which is a $2$-sheeted covering map. By \cite[Theorem 9.1]{Br}, we have a $2$-sheeted covering $\xi:G'=S^1 \to S^1$ and a \emph{unique} effective $G'$-action $\tilde{\mu}: G'\times \pi^{-1}(V_{\gamma})\to \pi^{-1}(V_{\gamma})$ such that the following diagram commutes:
\[
\xymatrix{
G'\times \pi^{-1}(V_{\gamma}) \ar[d]_{(\xi,\pi)} \ar[r]^{ \ \ \tilde{\mu}} &  \pi^{-1}(V_{\gamma}) \ar[d]^{\pi}\\
S^1\times V_{\gamma} \ar[r]_{\mu} &  V_{\gamma} }
\]

By applying this procedure to each $\gamma\in\Gamma$, we obtain an effective circle action on each element of the open covering $\mathcal{V}=\{\pi^{-1}(V_{\gamma})\}_{\gamma\in\Lambda}$ of $\tilde{X}_0$. 


\begin{lem}
\label{L:LIFTING}
The covering $\mathcal{V}$ determines a unique, effective, local $S^1$-action $\tilde{\Gamma}$ on $\tilde{X}$ covering $\Gamma$.
\end{lem}


\begin{proof}

Let $\alpha,\beta\in \Gamma$ be fibers such that $\pi^{-1}(V_{\alpha})\cap\pi^{-1}(V_{\beta})\neq \emptyset$. Therefore, $\pi(\pi^{-1}(V_{\alpha})\cap\pi^{-1}(V_{\beta}))\subseteq \pi(\pi^{-1}(V_{\alpha}))\cap\pi(\pi^{-1}(V_{\beta})) =V_{\alpha}\cap V_{\beta}$. Thus, $V_{\alpha}\cap V_{\beta}\neq\emptyset$.

Observe now that $V_{\alpha}\cap V_{\beta}$ is ``invariant" under the local $S^1$-action on $X$, i.e. for each $x\in V_{\alpha}\cap V_{\beta}$ there is exactly one fiber $\gamma_x\in \Gamma$ containing $x$ and, since $V_{\alpha}$ and $V_{\beta}$ are composed of fibers of $\Gamma$, $\gamma_x$ must be contained in both $V_{\alpha}$ and $V_{\beta}$. Therefore $V_{\alpha}\cap V_{\beta}$ is composed entirely of elements of $\Gamma$. 

Now, we consider $\pi^{-1}(V_{\alpha}\cap V_{\beta})=\pi^{-1}(V_{\alpha})\cap \pi^{-1}(V_{\beta})$. The lift of the local $S^1$-action on $V_{\alpha}\cap V_{\beta}$ to $\pi^{-1}(V_{\alpha}\cap V_{\beta})$ coincides with either the restriction of the one on $\pi^{-1}(V_{\alpha})$ or the one on $\pi^{-1}(V_{\beta})$ (by the uniqueness part of \cite[Theorem 9.1]{Br}). 

We define $\tilde{\Gamma}_0$ to be the set of fibers of the lifted local $S^1$-actions on each element of $\mathcal{V}$. By our previous observations $\tilde{\Gamma}_0$ is an effective, local $S^1$-action on $\tilde{X}_0$. By gluing (via fiber-preserving homeomorphisms) a collection of $3$-balls equipped with orthogonal $S^1$-actions with only one fixed point to $\tilde{X}_0$, we obtain an effective, local $S^1$-action $\tilde{\Gamma}$ on $\tilde{X}$.
\end{proof}

We now observe that the involution $\iota:\tilde{X}\to \tilde{X}$ commutes with $\tilde{\Gamma}$. As in the case of global $S^1$-actions, for each $\pi^{-1}(V_{\gamma})\in \mathcal{V}$,  the kernel of the two-sheeted cover $\xi:S^1\to S^1$ and the group of deck transformations of $\pi|_{\pi^{-1}(V_{\gamma})}$ coincide and are isomorphic to $\mathbb{Z}_2$ (cf. \cite[Section 4.1]{NZ}).
Therefore, $\iota$ commutes with $\tilde{\mu}$. Since this is the case for each element of $\mathcal{V}$, the involution $\iota$ commutes with $\tilde{\Gamma}$. 
\end{proof}



\begin{prop}
\label{L:RIEMANNIAN_METRIC}
There exists a Riemannian (orbifold) metric $g$ on $X$ such that the given topological local circle action is isometric with respect to this metric.
\end{prop}


\begin{proof}
We distinguish two cases, depending on whether or not $X$ is a topological manifold. 

Suppose first that $X$ is a topological manifold. Since $X$ is three-dimen\-sion\-al, it admits a (unique) smooth structure. 
Let $\Gamma=\{\gamma\}_{\gamma\in \Lambda}$ be the decomposition of $X$ into the fibers of the local circle action. For each $\gamma\in \Gamma$ we consider  a small tubular neighborhood $V_{\gamma}$ of $\gamma$ so that we have an effective, topological  circle action on $V_{\gamma}$. By the work of Orlik and Raymond \cite{OR,R}, the circle action on $V_\gamma$ is equivariantly homeomorphic to a smooth circle action. By \cite[Theorem 3.65]{AlexBett}, there exists a Riemannian metric $g_{\gamma}$ on each $V_\gamma$ such that the circle action is isometric. We now consider an equivariant smooth partition of unity $\{f_{\gamma}:V_{\gamma}\to [0,\infty)\}_{\gamma\in\Lambda}$ subordinated to the cover $\{V_{\gamma}\}_{\gamma\in\Lambda}$ of $X$. Then, we define a Riemannian metric $g$ on $X$ by $\sum_{\gamma\in\Lambda}f_{\gamma}g_{\gamma}$. The metric $g$ is invariant with respect to the local circle action since each $g_{\gamma}$ is invariant with respect to the circle action on each $V_{\gamma}$.  

We now consider the case where $X$ is not a topological manifold. As noted before the proposition, $X$ is
homeomorphic to a smooth orbifold,
 since it is a quotient of its double branched cover by an involution that is equivalent to a smooth involution (see \cite{GG1}). 

For each $\gamma\in \Gamma$ we consider  a small tubular neighborhood $V_{\gamma}$ of $\gamma$ so that we have an effective $S^1$-action
\[
\mu_{\gamma}: S^1\times V_{\gamma}\to V_{\gamma}.
\]	
As in the proof of Proposition~\ref{T:DBL_BRANCHED_COVER}, we consider the restriction $\pi^{-1}(V_{\gamma}) \to V_{\gamma}$ of the double branched covering to $V_{\gamma}$. Then, by \cite[Theorem 9.1]{Br}, we can lift the $S^1$-action on $V_{\gamma}$ to $\pi^{-1}(V_{\gamma})$. By an analogous procedure to that of \cite[Section 4.1]{NZ}, $\pi^{-1}(V_{\gamma})$ admits a Riemannian metric invariant with respect to the lifted circle action. We equip $V_{\gamma}$ with the Riemannian orbifold metric $g_{\gamma}$ induced from $\pi^{-1}(V_\gamma)$. 
As in the proof of 
Proposition~\ref{T:DBL_BRANCHED_COVER},
the involution commutes with the lifted action, and therefore, this metric is invariant with respect to the circle action on $V_{\gamma}$. We let $\mathcal{V}$ be the open cover of $X$ given by $\{V_{\gamma}\}_{\gamma\in\Lambda}$.  

We now consider a smooth (in the sense of orbifolds) equivariant partition of unity $\{f_{\gamma}:V_{\gamma}\to [0,\infty)\}_{\gamma\in\Lambda}$ subordinated to the cover $\mathcal{V}$. Then, we define a Riemannian orbifold metric $g$ on $X$ by $\sum_{\Lambda}f_{\gamma}g_{\gamma}$. The metric $g$ is invariant with respect to the local circle action since each $g_{\gamma}$ is invariant with respect to the circle action on each $V_{\gamma}$.
\end{proof}

We now show that the local circle action is isometric with respect to the distance function induced by the Riemannian metric constructed in Proposition~\ref{L:RIEMANNIAN_METRIC}. The proof is rather technical, but it essentially follows from the fact that shortest geodesics can be covered with finitely many invariant neighborhoods of the local circle action. We first make the following observation.


\begin{lem} 
\label{L:ACTIONS_S1_ON_S1}
Any two continuous free actions $\mu_1, \mu_2$ of $S^1$ on itself are equivalent and therefore any one of them is equivalent to the action given by complex multiplication.  
\end{lem}


\begin{proof}
Let $S^1_1$ and $S^1_2$ denote the two copies of $S^1$ on which $\mu_1$ and $\mu_2$ act, respectively. Each of the orbit spaces $S^1_1/S^1$ and $S^1_2/S^1$ is a point so we can identify them. We will denote this point by $pt$. Since the orbit space is a point, for $i=1,2$, there exists a cross-section $s_i:pt\to S^1_i$ to $\mu_i$, defined by choosing any point in $S^1_i$ as the image of $pt$. Each of the points $x\in S^1_1$ has a unique representation of the form $x=\mu_1(g,s_1(pt))$ and, analogously, each $y\in S^1_2$ is uniquely expressed as $y=\mu_2(h,s_2(pt))$.  We define $\Psi:S^1_1 \to S^1_2$ by letting $\Psi(s_1(pt))=s_2(pt)$ and extending ``equivariantly", i.e.
\[
\Psi(x)=\Psi(\mu_1(g,s_1(pt)):= \mu_2(g,\Psi(s_1(pt)))= \mu_2(g,s_2(pt)))
\]
This function is clearly continuous, equivariant and the same is true for its inverse, which is constructed analogously. Therefore $\mu_1$ is equivalent to $\mu_2$.
\end{proof}


\begin{prop}
\label{L:METRIC}
The given topological local circle action on $X$ is isometric with respect to the distance induced by the Riemannian metric $g$ of Proposition \ref{L:RIEMANNIAN_METRIC}. Therefore, there exists an Alexandrov metric on $X$ such that the given topological local circle action is isometric.
\end{prop}

\begin{proof}

We will show that in both the manifold and non-manifold cases the distance function $d$ induced on $X$ by $g$ 
is an Alexandrov metric with respect to which the given topological local circle action is isometric. 
Observe first that, since  $d$ is induced by a Riemannian (orbifold) metric and $X$ is compact, $d$ has curvature bounded below in the comparison sense. Hence $(X,d)$ is an Alexandrov space. Now we show that the local action is isometric with respect to $d$. The proof works in both cases with no modifications. 

Let $\gamma$ be a fiber of the local circle action on $X$ and $V_{\gamma}$ a small tubular neighborhood of it. Now we let $x,y\in V_\gamma$ and $\eta:[0,1]\to X$ be a geodesic joining $x$ and $y$. Suppose that $\eta([0,1])$ is fully contained in $V_{\gamma}$. Then, since the action of $S^1$ on $V_{\gamma}$ is isometric, for any $\theta\in S^1$, the curve $\theta\eta:[0,1]\to V_{\gamma}$ given by $(\theta\eta)(t)=\theta\eta(t)$ is a geodesic between $\theta x$ and $\theta y$,  and this is analogous to the global circle action case. Therefore $d(x,y)=d(\theta x,\theta y)$. 

We will now examine the case where $\eta([0,1])$ is not contained in $V_{\gamma}$. First we will assume that $x$ and $y$ are topologically regular points. Since the stratum of topologically regular points of $X$ is convex (see \cite[Sec.~1.8]{Per3}),  $\eta([0,1])$ will be contained in this stratum. In general, $\eta([0,1])\cap F$, $\eta([0,1])\cap E$ and $\eta([0,1])\cap SE$ might not be empty. For this reason, we divide the following analysis into several cases.\\

We first consider the case where
 $\eta([0,1])\subset R$, where $R\subset X$ is the stratum of  $R$-fibers of the local circle action on $X$. Since $\eta([0,1])$ is compact we can cover it with a finite number of open sets of the form $V_{\gamma_{\eta(t)}}$, where $\gamma_{\eta(t)}$ is the unique fiber of the local circle action going through $\eta(t)$. 
We will denote them by $V_i$, with $i=0,1,\ldots, n$. We will also denote the circle action on $V_i$ by $\mu_i$. We can take a partition $0=t_0<t_1<\cdots<t_{n+1}=1$ of $[0,1]$ such that $\eta([t_i,t_{i+1}])\subset V_i$. We may assume that $V_0$ and $V_n$ are contained in $V_\gamma$, so that the actions $\mu_0$ and $\mu_n$ are given by the action $\mu$ on $V_\gamma$.

Now, for each $\theta_0\in S^1$, we will construct a curve $\tilde{\eta}$ joining $\mu_0(\theta_0,x)=\mu(\theta_0,x)$ and 
$\mu_n(\theta_0,y) =\mu(\theta_0,y)$ with the same length as $\eta$. The curve $\tilde{\eta}$ is defined piecewise in the following way. We consider the curve
\[
\tilde{\eta}_0:[t_0,t_1]\to V_0
\]
\[
t\mapsto \mu_0(\theta_0,\eta|_{[t_0,t_1]}(t)).
\]
Now, note that $\eta(t_1)\in V_0\cap V_1$. Since the fibers of $\mu_0$ coincide with the fibers of $\mu_1$  on $V_0\cap V_1$, then $\eta(t_1)$ and $\mu_0(\theta_0,\eta(t_1))$ are in the same fiber with respect to $\mu_1$. Therefore, there exists $\theta_1\in S^1$ such that $\mu_1(\theta_1,\eta(t_1))=\mu_0(\theta_0,\eta(t_1))$. We define
\[
\tilde{\eta}_1:[t_1,t_2]\to V_1
\]
\[
t\mapsto \mu_1(\theta_1,\eta|_{[t_1,t_2]}(t)).
\]
We continue this process inductively, noting that for each $i\in\{1,\ldots, n\}$, $t_i\in V_{i-1}\cap V_i$ and there exists $\theta_i\in S^1$ such that $\mu_{i}(\theta_i,\eta(t_i))=\mu_{i-1}(\theta_{i-1},\eta(t_i))$. We then let the curve $\tilde{\eta}_i$ be given by
\[
\tilde{\eta}_i:[t_i,t_{i+1}]\to V_i
\]
\[
t\mapsto \mu_i(\theta_i,\eta|_{[t_i,t_{i+1}]}(t)).
\]
As a consequence of Lemma~\ref{L:ACTIONS_S1_ON_S1}, we have that $\tilde{\eta}_n(t_n)=\mu_n(\theta_0,y) = \mu(\theta_0,y)$. We define $\tilde{\eta}:[0,1]\to R$, given by
\[ 
\tilde{\eta}(t):= \tilde{\eta}_i(t) \ \mbox{ if } t\in [t_i,t_{i+1}],
\]
which is a  piecewise smooth curve with the same length as $\eta$. Therefore $d(x,y)\geq d(\mu_0(\theta_0,x), \mu_0(\theta_0,y))$. By an analogous procedure applied to a geodesic between $\mu_0(\theta_0,x)$ and $\mu_0(\theta_0,y)$, we obtain the opposite inequality which yields that $d(x,y)= d(\mu_0(\theta_0,x), \mu_0(\theta_0,y))$.\\

Now we consider the case in which $\eta([0,1])\cap (F\cup SF\cup E\cup SE)\neq \emptyset$. First, we make the following observation. In the case that $\eta(0)$ or $\eta(1)$ are not in $R$, for any $\varepsilon>0$ we can slightly perturb $\eta([0,1])$ inside $V_{\gamma}$ to obtain a curve $\bar{\eta}:[0,1]\to X$ which coincides with $\eta([0,1])$ outside $V_{\gamma}$ such that 
\[
\left\lvert\,\mathrm{Length}(\bar{\eta}) - \mathrm{Length}(\eta)\,\right\rvert \leq  \varepsilon
\]
and with its endpoints contained in $R$. Using the previous case in which we assumed that $\eta([0,1])\subset R$, we will show that for any such curve $\bar{\eta}$ and any $\theta\in S^1$ there exists a curve of the same length joining $\mu(\theta, \bar{\eta}(0))$ and $\mu(\theta, \bar{\eta}(1))$. Assuming this, we choose a sequence $\varepsilon_i\to 0$ and consider the corresponding sequence of curves $\bar{\eta}_i$. Then, up to a subsequence, the induced curves joining  $\mu(\theta, \bar{\eta}_i(0))$ and $\mu(\theta, \bar{\eta}_i(1))$ converge to a curve joining $\mu(\theta, \eta(0))$ and $\mu(\theta, \eta(1))$ having the same length as $\eta$. Therefore, in the following we assume that the endpoints of $\eta$ are contained in $R$. 

Since the $F$- and $E$-components are codimension $2$ submanifolds of $X$ and are isolated, for each $\varepsilon>0$, we may perturb $\eta$ to obtain a curve 
$\hat{\eta}:[0,1]\to X$ joining $x$ and $y$ contained in the topologically regular stratum and satisfying that
\begin{itemize}
\item[(i)] $\hat{\eta}([0,1])\cap (F\cup E)=\emptyset$,
\item[(ii)] $\hat{\eta}([0,1])\cap SE = \eta([0,1])\cap SE$, and
\item[(iii)] $ \left\lvert\, \mathrm{Length}(\hat{\eta}) - \mathrm{Length}(\eta)\,\right\rvert \leq \varepsilon$.
\end{itemize}

Let us now consider the case in which $\hat{\eta}$ intersects the $SE$-stratum. 
Since there are finitely many connected components of the $SE$-stratum, and each component is a $2$-dimensional manifold, we can slightly perturb the curve $\hat{\eta}$ to obtain a curve with the same endpoints which intersects each component of the $SE$-stratum transversally at at most a finite number of points. Moreover, we may assume that the length of the perturbed curve is arbitrarily close to the length of $\hat{\eta}$.  Abusing notation, we will also call this new curve $\hat{\eta}$. Let $\{\hat{\eta}(s_i)\}= \hat{\eta}([0,1])\cap SE$. We can further assume that on a sufficiently small interval $[a_i,b_i]\subset [0,1]$ with $a_i<s_i<b_i$, the curve $\hat{\eta}$ is a geodesic. 

Therefore, 
for any $\varepsilon>0$, there exists a curve $\tilde{\eta}:[0,1]\to X$ such that
\begin{itemize}
\item[(i)] $\tilde{\eta}([0,1])\setminus \{\tilde{\eta}(s_i)\}_{i=1}^m \subset R$,
\item[(ii)] $\tilde{\eta}$ only has a finite number of points $\{\tilde{\eta}(s_i)\}_{i=1}^m$ in $SE$ and $\tilde{\eta}$ is a geodesic on sufficiently small intervals $[a_i,b_i]$ with $a_i<s_i<b_i$, and
\item[(iii)] $ \left\lvert\,\mathrm{Length}(\tilde{\eta}) - \mathrm{Length}(\eta)\,\right\rvert \leq \varepsilon$.
\end{itemize}

Let us fix $\varepsilon>0$ and $\theta_0\in S^1$. We cover $\tilde{\eta}$ with a finite number of invariant neighborhoods $V_j$ of the local action in such a way that for each $i$,  $\tilde{\eta}([a_i,b_i])$ is fully contained in exactly one of the neighborhoods $V_j$, which we denote by $V_{j_i}$. Let $\mu_j$ be the isometric circle action on $V_j$. Let us consider $\tilde{\eta}([0,a_1])$, which is covered by $\{V_j\}_{j=0}^{j_1}$. By the same procedure as in the case in which $\eta$ was contained in $R$, we obtain a curve $\rho_1:[0,a_1]\to X$ joining $\mu_0(\theta_0,x)$ and $\mu_{j_1}(\theta_{j_1}, \tilde{\eta}(a_1))$ for some $\theta_{j_1}\in S^1$ such that 
\[
\mathrm{Length}(\tilde{\eta}\,|\,{[0,a_1]})= \mathrm{Length}(\rho_1\,|\,{[0,a_1]}).
\]

Since $\tilde{\eta}(b_1) \in V_{j_1}$ we can use the same $\theta_{j_1}$ to begin the process on $\tilde{\eta}\,|\,{[b_1,a_2]}$, obtaining a curve $\rho_2:[b_1,a_2]\to X$ joining $\mu_{j_1}(\theta_{j_1}, \tilde{\eta}(b_1))$ and $\mu_{j_2}(\theta_{j_2},\tilde{\eta}(a_2))$ for some $\theta_{j_2}\in S^1$ such that  
\[
\mathrm{Length}(\tilde{\eta}\,|\,[b_1,a_2])= \mathrm{Length}(\rho_2\, | \,[b_1,a_2]).
\]
Continuing in this way we obtain a finite number of curves $\rho_i:[b_{i-1}, a_i]\to X$, $i=1,\ldots,m$, (we abuse notation and set $b_0=0$), such that
\[
\mathrm{Length}(\rho_i)= \mathrm{Length}(\tilde{\eta}\, | \,[b_{i-1}, a_i])
\] 
for each $i=1,\ldots,m$. 
As a consequence of Lemma \ref{L:ACTIONS_S1_ON_S1}, the curve $\rho_m$ joins $\mu_{j_{m-1}}(\theta_{j_{m-1}},\tilde{\eta}(b_m))$ and $\mu_0(\theta_0,y)$. Furthermore, since the $\tilde{\eta} \, |\,[a_i,b_i]$ are geodesics fully contained in $V_{j_i}$, they are sent to geodesics under the action of $\theta_{j_i}$. Therefore, concatenating the curves $\rho_i$ and these geodesics, we obtain a piecewise smooth curve $\rho:[0,1]\to X$ joining $\mu_0(\theta_0,x)$ and $\mu_0(\theta_0,y)$ such that 
\[
\mathrm{Length}(\tilde{\eta})= \mathrm{Length}(\rho).
\] 

This implies that $d(x,y)\leq d(\mu_0(\theta_0,x),\mu_0(\theta_0,y))+\varepsilon$.
Then, it suffices to take a sequence $\varepsilon_k\to 0$, consider the associated curves $\tilde{\eta}_k$ and $\rho_k$ to find that 
\[
d(\mu_0(\theta_0,x),\mu_0(\theta_0,y))\geq \lim_{k\to\infty}\mathrm{Length}(\tilde{\eta}_k)=d(x,y).
\] 
Doing the procedure for curves joining $\mu_0(\theta_0,x)$ and $\mu_0(\theta_0,y)$ we get the reverse inequality.  We conclude, then, that the local circle action is isometric on $X_0$.

Finally in the case in which either $x$ or $y$ is a topologically singular point, the result follows by the density of the set  $X_0$ of topologically regular points.
\end{proof}


\section{Topological and equivariant classification}
\label{S:CLASS_GENERAL}
In this section we  prove Theorem~\ref{THM:INVARIANTS}. We first  prove the following technical lemma.
The term \textit{cross-section} (or simply \textit{section}) will be used to denote a map $\mathsf{q}:X^*\to X$ such that $\mathsf{q} \circ \mathsf{p} : X\to X$ is the identity,  as well  as its image $\mathsf{q}(X^*)\subset X$.


\begin{lem}
\label{L:SECTION}
Let $X$ be a closed, connected Alexandrov $3$-space admitting an effective and isometric local circle action without exceptional fibers and $F\neq\emptyset$. Then there exists a cross-section to the fiber map. 
\end{lem}


\begin{proof}
Let $\mathsf{p}_0:X_0\rightarrow X_0^{*}$ be the restriction of the fiber map to the stratum of $R$-fibers.
Then, by analogous arguments to those of the proof of \cite[Lemma 2]{R}, there exists a cross-section $\mathsf{q}:X_0^{*}\rightarrow X_0$. We let $\mathsf{q}_j$ denote the restriction of $\mathsf{q}$ to the $j$-th boundary component of $X_0^{*}$.
 The $\mathsf{q}_j$ that lie on boundary components corresponding to simple $F$-blocks and simple $SE$-blocks determine extensions of $\mathsf{q}$ to said blocks as in \cite{R}. Extensions of $\mathsf{q}$ to twisted $F$-blocks and twisted $SE$-blocks are constructed analogously. More explicitly, $\mathsf{q}$ determines curves on the Klein bottle boundary components of $X_0$ and we can extend these curves radially to obtain the desired cross-sections. An extension to $\mathsf{q}$ over the simple $SF$-blocks is constructed as in \cite[Theorem 4.1]{NZ}. In order to extend $\mathsf{q}$ to the twisted $SF$-blocks we do the following. We decompose the fiber space of a twisted $SF$-block into two annuli. One of these annuli will have only principal orbits, except at one of its boundary components, where the $SF$-, $F$- and $SE$-points lie. The other annulus is the fiber space of the subset $B_F^0:=K\times [t_0,1]$ of a twisted $F$-block for some $0<t_0<1$. It is clear that $B_F^0$ has a section $\tilde{\mathsf{q}}$ extending $\mathsf{q}$. Therefore, we can extend $\tilde{\mathsf{q}}$ to the whole twisted $SF$-block as in \cite[Theorem 4.1]{NZ}. 
\end{proof}


\begin{rem}
\label{REM:SF_BLOCKS}
Using the cross-section obtained in the previous lemma it is possible to prove that the effective, isometric $S^1$-action on a twisted $SF$-block, equipped with the Riemannian orbifold metric of non-negative curvature is unique up to equivariant homeomorphisms (cf. \cite[Lemma 3]{R},  \cite[Corollary 4.5]{NZ}).
\end{rem}


\begin{proof}[Proof of Theorem \ref{THM:INVARIANTS}] 
We begin by noting that if $X$ does not have any topologically singular points then the result reduces to the classification of effective local circle actions for $3$-manifolds (\cite{F} and \cite{OR2}) in combination with the existence of an invariant Alexandrov metric (see Proposition \ref{L:METRIC}).

Thus, we assume that $s>0$. At the beginning of Section~\ref{S:CONSTRUCTIONS}, we indicated how to obtain a topological space with an effective, local circle action with prescribed invariants. By Proposition \ref{L:METRIC}, this space is an Alexandrov space and the local circle action is isometric. Conversely, given an Alexandrov space with an effective and isometric circle action we can ``read off'' the invariants from the action and from $X$. This proves the first part of the Theorem. 

Now we prove the second part of the theorem. We assume for now that there are no exceptional fibers. Consider the unique $3$-manifold $M$ with the local $S^1$-action determined by $\left\{b; \varepsilon, g, (f+s,k_1+k_3), (t,k_2)\right\}$ as in \cite[Theorem 2]{F}. Note that since $s>0$, $M$ has at least $s$ boundary $F$-components. 
Here, $s-k_3$ of these $F$-components, which we denote by $Q_i$ with $i=1,\ldots, s-k_3$, correspond to simple $F$-blocks.  The remaining $k_3$, denoted by $P_j$, with $j=1,\ldots,k_3$, belong to twisted $F$-blocks.  We now let $R_h$  be the equivariant connected sum of $h$  copies of $\Susp(\mathbb{R}P^2)$. 
 For  each $0\leq i \leq s-k_3$, we successively perform an equivariant connected sum of $M$ and $R_{r_i/2}$ centered at an $F$-fiber in $Q_i$ and similarly, for every $0\leq j \leq k_3$, we successively carry out an equivariant connected sum of $M$ with $R_{q_j/2}$, centered at an $F$-fiber on $P_j$.  We let $X'$ denote the resulting space $M\#R_{r_1/2}\#R_{r_2/2}\# \cdots\#R_{r_{s-k_3}/2}\#R_{q_1/2}\#R_{q_2/2}\# \cdots\#R_{q_{k_3}/2}$ and we observe that it has the fiber space 
\[
(X')^{*}\cong M^{*}\#R^*_{r_1/2}\#R^*_{r_2/2}\# \cdots\#R^*_{r_{s-k_{3}}/2}\#R^*_{q_1/2}\#R^*_{q_2/2}\# \cdots\#R^*_{q_{k_{3}}/2}.
\] 
By Lemma \ref{L:SECTION}, there exists a cross-section $(X')^*\to X'$. Using this cross-section and the methods of \cite[Lemma 3]{R} and  \cite[Corollary 4.5]{NZ} we obtain an equivariant equivalence $X'\cong X$. Thus, we conclude that 
\[
X \cong M\#\underbrace{\Susp(\Real P^2)\# \cdots \# \Susp(\Real P^2)}_{r \text{ summands}}.
\] 

Finally, \cite[ Lemma $6$, Theorems $2a$, $2b$ ]{R} extend to Alexandrov spaces naturally. Therefore, in the case where $X$ has exceptional fibers, the equivalence still holds.    
\end{proof}


\section{Local circle actions and collapse}
\label{S:COLLAPSE}

The collapse of closed, three dimensional Alexandrov spaces was studied by Mitsuishi and Yamaguchi in \cite{MY}. We show in this section that, when the limit space is two-dimensional, some of these spaces admit local circle actions and that the collapse occurs along the orbits of the action.

Let $X$ be a closed, three-dimensional Alexandrov space that collapses with lower curvature bound and upper diameter bound to a two-dimensional space. Mitsuishi and Yamaguchi showed in \cite{MY} that $X$ decomposes as a union of certain blocks. We recall the definition of those that admit circle actions, the so-called \textit{generalized solid tori} and \textit{generalized solid Klein bottles}. 
 
 Consider the families of surfaces in $\mathbb{R}^3$ given by
\[
A(t)=\{ v=(x,y,z) \mid x^2 + y^2-z^2=t^2 \ \ \text{and} \ \  | z|\leq 1  \},
\]
\[
B(t)= \{ v=(x,y,z) \mid x^2 + y^2-z^2=-t^2 \ \ \text{and} \ \  x^2 + y^2 \leq 1  \}.
\]
The $A(t)$ are one-sheeted hyperboloids (and a cone for $t=0$), and $B(t)$ are two-sheeted hyperboloids. We set
\[ 
D(t):= \begin{cases} 
      	  A(t)/\mathbb{Z}_2 & \text{for} \ \ t> 0, \\
          B(t)/\mathbb{Z}_2 & \text{for} \ \ t\leq 0,       
          \end{cases}
\]
where $\mathbb{Z}_2$ acts on $A(t)$ and $B(t)$ by the antipodal map. Therefore, $D(t)$ is homeomorphic to a M\"{o}bius band $\mathrm{Mo}$ if $t>0$ and a disk $D^2$ if $t\leq 0$. Observe that $ \partial D(t) \cong \mathbb{S}^1$. Therefore
\[
\bigcup_{t\in[-1,1]}\partial D(t) \approx \mathbb{S}^1\times [-1,1].
\]

We  have that $\bigcup_{t\geq0}(A(t)\cup B(-t))/\mathbb{Z}_2$ is homeomorphic to the closed cone $K_1(\mathbb{R}P^2)$ and 
 we have a projection map $\pi: K_1(\mathbb{R}P^2)\rightarrow [-1,1]$  given by $\pi(D(t))=t$.

For $N\geq 1$ we regard the circle $\mathbb{S}^1$ as the interval $[0,2N]$ with its endpoints identified. Let $I_j = [j-1,j] \subset [0,2N]$ for $j=1,2,\ldots, 2N$. Now we take a family of spaces $\{B_j\}_{j=1}^{2N}$ such that for each $j$, $B_j$ is isometric to a $K_1(\mathbb{R}P^2)$ constructed similarly to the one above and let $\pi_j:B_j\to I_j$ be a projection similar to the one defined previously. 

These projections coincide at each $\{j\}=I_j\cap I_{j+1}$, that is, $\pi_j^{-1}(j)=\pi_{j+1}^{-1}(j)$ for all $j=1,2,\ldots, 2N$. We glue the $B_j$ with homeomorphisms along the $\pi_j^{-1}(j)$ to obtain a space $Y= \bigcup_{j=1}^{2N}B_j$ which has a ``fibration'' $\pi: Y\rightarrow \mathbb{S}^1$ over $\mathbb{S}^1$ given by $\pi(\pi_j^{-1}(t))=t$. Observe that $Y$ has $2N$ topologically singular points since each $B_j$ has one topologically singular point (corresponding to  the vertex of $D(0)$).

 The restriction $\pi|_{\partial Y}:\partial Y\to \mathbb{S}^1$ is an $\mathbb{S}^1$-bundle over $\mathbb{S}^1$. If $\partial Y$ is a torus, then $Y$ is called a \emph{generalized solid torus of type $N$}. If $\partial Y$ is a Klein bottle, then $Y$ is called a \emph{generalized solid Klein bottle of type $N$}. The spaces $D^2\times S^1$ and $\mathrm{Mo}\times S^1$ are generalized solid tori of type $0$, while $D^2\tilde{\times}S^1$ and $\mathrm{Mo}\tilde{\times}S^1$ are generalized solid Klein bottles of type $0$.


\begin{lem}
\label{P:SOLID_TORUS=SF-BLOCK}
Let $Y$ be a generalized solid torus or a generalized solid Klein bottle of type $N$. Then $Y$ admits an effective, topological local circle action and the following hold:
\begin{itemize}
	\item[1.] If $Y$ is a generalized solid torus and $N>0$, then it is equivalent to a simple $SF$-block.
	\item[2.] If $Y$ is a generalized solid Klein bottle and $N>0$, then it is equivalent to a twisted $SF$-block.
\end{itemize}
\end{lem}


\begin{proof}
We identify $\mathbb{R}^3$ with $\mathbb{C}\times \mathbb{R}$. Let us denote the elements of $D(t)$ by $[w,s]$, where $(w,s)\in A(t)\subset \mathbb{C}\times \mathbb{R}$ if $t>0$ and $(w,s)\in B(t)\subset \mathbb{C}\times \mathbb{R}$ if $t\leq 0$.  Then, for each $t\in [-1,1]$, we have an effective topological circle action
\begin{eqnarray*}
S^1 \times D(t) & \longrightarrow & D(t) \\
(z, [w, s]) & \longmapsto & [zw,s],
\end{eqnarray*} 
where $zw$ is complex multiplication. 
Observe that $\partial D(t)$ is an invariant subset of this action. Therefore, this action extends to an effective topological circle action on $B_j$. We conclude that $Y$ admits effective topological local circle actions, defined by the condition that their restriction to each $B_j$ coincides with the action defined above. Then, by inspecting the list of non-manifold blocks in Section \ref{S:Orbit_types}, we obtain the result. 
\end{proof}

Let $\{X_i\}_{i=1}^{\infty}$ be a sequence of closed Alexandrov $3$-spaces with curvature $\geq -1$ and $\mathrm{diam}\leq D$ converging to an Alexandrov surface $X$ with non-empty boundary. According to \cite[Theorem 1.5]{MY}, for $i$ large enough, $X_i$ is homeomorphic to a union of  generalized solid tori or solid Klein bottles (which are considered as fiberings over the boundary components of $X$) and a \textit{generalized Seifert fiber space} (i.e.\ a Seifert fiber space possibly having singular interval fibers (see \cite[Definition 2.48]{MY} for the precise definition). Small tubular neighborhoods of these singular fibers are homeomorphic to the space $B(pt)=D^2\times S^1/\alpha$ where $\alpha$ is the isometric involution given by $\alpha(e^{i\theta},x)=(e^{-i\theta},-x)$. The singular fiber corresponds to $S^1\times \{0\}/\alpha$. 

The collapse around singular fibers is the one obtained by shrinking the $S^1$-factor, and the collapsed space is homeomorphic to $K_1(S^1)$. The topologically singular points project to the  the vertex of $K_1(S^1)$, which is in the interior of the limit space $X$. Therefore the collapse about singular fibers cannot be the one along the fibers of an effective and isometric circle action, since in this case topologically singular points project to points in the boundary of the limit space $X$. Generalized Seifert fiber spaces without $B(pt)$ fibers are usual Seifert fiber $3$-manifolds, which admit effective local circle actions. Therefore, by Lemma~\ref{P:SOLID_TORUS=SF-BLOCK}, $X_i$ admits a local circle action compatible with the collapse, if and only if its decomposition does not contain $B(pt)$ pieces. 
We obtain in this way the following corollary:


\begin{cor}
\label{C:COLLAPSE}
Let $\{X_i\}_{i=1}^{\infty}$ be a sequence of closed Alexandrov $3$-spaces with curvature $\geq -1$ and $\mathrm{diam}\leq D$ converging to an Alexandrov surface $X$ (possibly with boundary). Further assume that, for $i$ large enough, $X_i$ does not have any singular fibers of type $B(pt)$. Then, for $i$ large enough, $X_i$ is homeomorphic to an Alexandrov space with an effective and isometric local circle action and the collapse $X_i\to X$ occurs along the fibers of the local action. In particular, $X_i$ is homeomorphic to one of the spaces in Theorem~\ref{THM:INVARIANTS}.
\end{cor}

\bibliographystyle{amsplain}


\end{document}